\documentclass[a4paper,twoside,fleqn,notitlepage]{article}%
\usepackage{amsmath}
\usepackage{amsfonts}
\usepackage{amssymb}
\usepackage{graphicx}%
\providecommand{\U}[1]{\protect\rule{.1in}{.1in}}
\newtheorem{theorem}{Theorem}

\newtheorem{definition}[theorem]{Definition}

\newtheorem{lemma}[theorem]{Lemma}

\newtheorem{proposition}[theorem]{Proposition}
\newtheorem{remark}[theorem]{Remark}

\newenvironment{proof}[1][Proof]{\noindent \textbf{#1.} }{\  \rule{0.5em}{0.5em}}
\begin{document}

\title{ Study of new class of $q-$fractional integral operator}
\author{M. Momenzadeh, N. I. Mahmudov\\Near East University\\Eastern mediterranean university \\Lefkosa, TRNC, Mersin 10, Turkey \\Famagusta, North Cyprusvia Mersin 10 Turkey \\Email: mohammad.momenzadeh@neu.edu.tr\\nazim.mahmudov@emu.edu.tr\\\ \ \ \ \ \ \ \ \ }
\maketitle

\begin{abstract}
We study the new class of $q-$fractional integral operator. In the aid of
iterated Cauchy integral approach to fractional integral operator, we applied
$t^{p}f(t)$ instead of $f(t)$ in these integrals and with parameter $p$ a new
class of $q$-fractional integral operator is introduced. Recently the
$q$-analogue of fractional differential integral operator is
studied.\cite{agarwal}\cite{mansour}\cite{rajkovic}\cite{Zhang}\cite{11}%
\cite{12}All of these operators are $q-$analogue of Riemann fractional
differential operator. The new class of introduced operatorgeneralize all
these defined operator and can be cover the $q$-analogue of Hadamard
fractional differential operator. Some properties of this operator is investigated.


\end{abstract}

\section{Introduction}

\bigskip Fractional calculus has a long history and has recently gone through
a period of rapid development. When Jackson (1908)\cite{jack} defined
$q$-differential operator, $q-$calculus became a bridge between mathematics
and physics. It has a lot of applications in different mathematical areas such
as combinatorics, number theory, basic hypergeometric functions and other
sciences: quantum theory, mechanics, theory of relativity, capacitor theory,
electrical circuits, particle physics, viscoelastic, electro analytical
chemistry, neurology, diffusion systems, control theory and statistics; The
$q-$Riemann-Liouville fractional integral operator was introduced by Al-Salam
\cite{alsalam}, from that time few $q$-analogues of Riemann operator were
studied. \cite{rajkovic}\cite{gamma kac}\cite{agarwal}\cite{mansour}

On the other hand, recent studies on fractional differential equations
indicate that a variety of interesting and important results concerning
existence and uniqueness of solutions, stability properties of solutions, and
analytic and numerical methods of solutions for these equations have been
obtained, and the surge for investigating more and more results is underway.
Several real world problems were modeled in the aid of using fractional
calculus. Nowadays, fractional-order differential equations can be traced in a
variety of applications such as diffusion processes, biomathematics,
thermo-elasticity, etc.\cite{applic}. However, most of the work on the topic
is based on Riemann-Liouville, and Caputo-type fractional differential
equations. $q$-analogue of this operator is defined \cite{alsalam} and
application of this operator is investigated.\cite{mansour}\cite{agarwal}%
\cite{Zhang} Another kind of fractional derivatives that appears side by side
to Riemann-Liouville and Caputo derivatives in the literature is the
fractional derivative due to Hadamard, introduced in 1892 \cite{hadamard},
which contains logarithmic function of arbitrary exponent in the kernel of the
integral appearing in its definition. In a paper from 1751, Leonhard Euler
(1707--1783) introduced the series which can be assumed as a $q-$analogue of
logarithm function. \cite{log}Since that time, a lot of mathematician have
tried to define $q-$logarithm function. Because of dificulty of working with
these function, there is no $q-$analogue of Hadamard fractional differential
integral operator. Hadamard-type integrals arise in the formulation of many
problems in mechanics such as in fracture analysis. For details and
applications of Hadamard fractional derivative and integral, we refer the
reader to a new book that gathered all of these applications.\cite{bashir}

In this paper, new $q$-integral operator is introduced, Some properties and
relations are investigated. In fact, a parameter is used to generalize the
Riemann operator and a new class of $q-$fractional difference operator is
introduced. In the first section, let us introduce some familar concepts of
q-calculus. Most of these definitions and concepts are available in \cite{kac}
and \cite{ernst}. We use $[n]_{q}$ as a $q-$analogue of any complex number.
Naturally, we can define $\left[  n\right]  _{q}!$ as
\[
\left[  a\right]  _{q}=\frac{1-q^{a}}{1-q}\ \ \ \left(  q\neq1\right)
;\ \ \ \left[  0\right]  _{q}!=1;\ \ \ \ \left[  n\right]  _{q}!=\left[
n\right]  _{q}\left[  n-1\right]  _{q}\ \ \ \ n\in\mathbb{N},\ \ a\in
\mathbb{C}\text{ }.
\]
The $q$-shifted factorial and $q$-polynomial coefficient are defined by
\begin{align*}
\left(  a;q\right)  _{0}  &  =1,\ \ \ \left(  a;q\right)  _{n}=%
{\displaystyle\prod\limits_{j=0}^{n-1}}
\left(  1-q^{j}a\right)  ,\ \ \ n\in\mathbb{N},\\
\left(  a;q\right)  _{\infty}  &  =%
{\displaystyle\prod\limits_{j=0}^{\infty}}
\left(  1-q^{j}a\right)  ,\ \ \ \ \left\vert q\right\vert <1,\ \ a\in
\mathbb{C}.
\end{align*}%
\[
\left(
\begin{array}
[c]{c}%
n\\
k
\end{array}
\right)  _{q}=\frac{\left(  q;q\right)  _{n}}{\left(  q;q\right)
_{n-k}\left(  q;q\right)  _{k}},
\]
In the standard approach to $q-$calculus, two exponential function are used
more.The first and second identities of Euler leads to two $q-$exponentials
function as follow%

\begin{align*}
e_{q}\left(  z\right)   &  =\sum_{n=0}^{\infty}\frac{z^{n}}{\left[  n\right]
_{q}!}=\prod_{k=0}^{\infty}\frac{1}{\left(  1-\left(  1-q\right)
q^{k}z\right)  },\ \ \ 0<\left\vert q\right\vert <1,\ \left\vert z\right\vert
<\frac{1}{\left\vert 1-q\right\vert },\ \ \ \ \ \ \ \\
E_{q}(z)  &  =e_{1/q}\left(  z\right)  =\sum_{n=0}^{\infty}\frac{q^{\frac
{1}{2}n\left(  n-1\right)  }z^{n}}{\left[  n\right]  _{q}!}=\prod
_{k=0}^{\infty}\left(  1+\left(  1-q\right)  q^{k}z\right)
,\ \ \ \ \ \ \ 0<\left\vert q\right\vert <1,\ z\in\mathbb{C},
\end{align*}
Let for some $0\leq\alpha<1,$ the function $|f(x)x^{\alpha}|$ is bounded on
the interval $\left(  0,A\right]  $, then Jakson integral defines as
\cite{kac}%

\[
\int\mathit{f}(x)d_{q}x=(1-q)x\sum_{i=0}^{\infty}q^{i}f(q^{i}x)
\]
converges to a function $F(x)$ on $\left(  0,A\right]  ,$ which is a
$q-$antiderivative of $f(x)$. Suppose $0<a<b$, the definite $q-$integral is
defined as%

\begin{align*}
\int\limits_{0}^{b}\mathit{f}(x)d_{q}x  &  =(1-q)b\sum_{i=0}^{\infty}%
q^{i}f(q^{i}b)\\
\int\limits_{a}^{b}\mathit{f}(x)d_{q}x  &  =\int\limits_{0}^{b}\mathit{f}%
(x)d_{q}x-\int\limits_{0}^{a}\mathit{f}(x)d_{q}x
\end{align*}

$q-$analogue of integral by part can be written as%

\begin{equation}
\int\limits_{a}^{b}g(x)D_{q}\mathit{f}(x)d_{q}x=\left(
g(b)f(b-g(a)f(a)\right)  -\int\limits_{a}^{b}f(qx)D_{q}\mathit{g}(x)d_{q}x
\end{equation}

In addition, we can interchange the order of double q-integral by%

\begin{equation}
\int\limits_{0}^{x}\int\limits_{0}^{v}\mathit{f}(s)d_{q}sd_{q}v=\int
\limits_{0}^{x}\int\limits_{qs}^{x}\mathit{f}(s)d_{q}vd_{q}s
\end{equation}

Actually this interchange of order is true, since%

\begin{equation}
\int\limits_{0}^{x}\int\limits_{qs}^{x}\mathit{f}(s)d_{q}vd_{q}s=\int
\limits_{0}^{x}\left(  x-qs\right)  f(s)d_{q}s=x(1-q)\sum_{i=0}^{\infty}%
q^{i}f(q^{i}x)\left(  x-q^{i+1}x\right)  =x^{2}(1-q)^{2}\sum_{i=0}^{\infty
}q^{i}f(q^{i}x)\left(  \sum_{j=0}^{\infty}q^{j}\right)
\end{equation}

In addition, the left side can be written as%

\begin{equation}
\int\limits_{0}^{x}\int\limits_{0}^{v}\mathit{f}(s)d_{q}sd_{q}v=x(1-q)\sum
_{i=0}^{\infty}q^{i}\int\limits_{0}^{xq^{i}}f(s)d_{q}s=x^{2}(1-q)^{2}%
\sum_{i=0}^{\infty}\sum_{j=0}^{\infty}q^{i+2j}f(q^{i+j}x)
\end{equation}

\bigskip Let $i+j=t$ to see that this releation is true. $q-$analogue of Gamma
function is defined as%

\[
\Gamma_{q}(t)=\int\limits_{0}^{\frac{1}{1-q}}\mathit{x}^{t-1}E_{q}(-qx)d_{q}x
\]

Several $q$-exponential functions for different purposes were defined.
\cite{unification}Natural question is appeard about this definition of
q-exponential function, Why did we use $E_{q}(x)$ in the definition of q-Gamma
function? Next lemma answer this question.

\begin{lemma}
For given q-Gamma function $\Gamma_{q}(t+1)=\left[  t\right]  _{q}\Gamma
_{q}(t)$ and this relation is true only where we use $E_{q}(-qx)$ instead of
q-exponential function.
\end{lemma}

\begin{proof}
Use q-integral by part and the fact that $D_{q}$\bigskip$(E_{q}(-x))=-E_{q}%
(-qx)$ to find the recurrence formula for q-gamma function.%

\[
\Gamma_{q}(t+1)=-\int\limits_{0}^{\frac{1}{1-q}}\mathit{x}^{t}d_{q}%
E_{q}(-x)=0+\left[  t\right]  _{q}\int\limits_{0}^{\frac{1}{1-q}}%
\mathit{x}^{t-1}E_{q}(-qx)d_{q}x=\left[  t\right]  _{q}\Gamma_{q}(t)
\]

Now if we assume that $f(x)=E_{q}(x)$, to have the above property, function
$f(x)$ should satisfies the following q-difference equation%

\begin{equation}
D_{q}\bigskip(f(q^{-1}x))=-f(x)
\end{equation}

If we apply q-derivative formula for this, we can see that $f(x)\left[
1+(q-1)x\right]  =f(q^{-1}x).$ If we rewrite this equation for $f(q^{-j}x)$
and tend $j$ to infinity, then we have%

\begin{equation}
f(x)=\frac{f(0)}{\prod_{k=0}^{\infty}\left(  1+\left(  q-1\right)
q-^{k}x\right)  }=f(0)e_{q^{-1}}\left(  x\right)  =f(0)E_{q}(x)
\end{equation}

Which shows the uniqueness of $E_{q}(-qx)$ in this definition.
\end{proof}

\bigskip Actually, the authors in \cite{gamma kac} define another version of
$q$-Gamma function. In that paper, Authors defined $q$-Gamma function such
that the classic results satisfied. This definition is based on $e_{q}\left(
x\right)  $ and as we mentioned it in the last lemma, the recurence formula is
not given the same terms like in lemma. In fact,%

\[
\widetilde{\gamma}_{q}^{(A)}(t+1)=q^{-t}\left[  t\right]  _{q}\widetilde
{\gamma}_{q}^{(A)}(t)
\]

Where $\widetilde{\gamma}_{q}^{(A)}(t)=\int\limits_{0}^{\frac{\infty}{A\left(
1-q\right)  }}\mathit{x}^{t-1}e_{q}(-x)d_{q}x$. The limit boundary of
integration is changed also and author used the interesting function as a kind
of weight and make the equivalent definition for $q$-gamma function.

$q$-shifted factorial may extend to the following definition%

\begin{equation}
(x-a)^{\left(  \alpha\right)  }=x^{\alpha}\prod_{k=0}^{\infty}\frac{\left(
1-\frac{x}{a}q^{k}\right)  }{\left(  1-\frac{x}{a}q^{k+\alpha}\right)  }%
=\frac{x^{\alpha}\left(  \frac{x}{a};q\right)  _{\infty}}{\left(  q^{\alpha
}\frac{x}{a};q\right)  _{\infty}}%
\end{equation}

We can rewrite the $q$-Gamma function by using this definition as\cite{gamma
kac}
\begin{equation}
\Gamma_{q}(t)=\frac{(1-q)^{\left(  t-1\right)  }}{(1-q)^{t-1}}%
\end{equation}

Let us generalize definition (7) to the following form%

\begin{equation}
(x-y)_{q^{p+1}}^{\left(  \alpha\right)  }=x^{\alpha}\prod_{k=0}^{\infty}%
\frac{\left(  x-y\left(  q^{p+1}\right)  ^{k}\right)  }{\left(  x-y\left(
q^{p+1}\right)  ^{k+\alpha}\right)  }=\frac{x^{\alpha}\left(  \frac{y}%
{x};q^{p+1}\right)  _{\infty}}{\left(  q^{\alpha\left(  p+1\right)  }\frac
{y}{x};q^{p+1}\right)  _{\infty}}%
\end{equation}

In addition, we define another version of $q-$analogue of exponent as%

\begin{align}
\left[  p+1\right]  ^{\left(  \alpha\right)  }  &  =\prod_{k=1}^{\infty}%
\frac{\left[  p+1\right]  _{q^{k}}}{\left[  p+1\right]  _{q^{k+\alpha}}}%
=\frac{\left(  1-q^{p+1}\right)  \left(  1-q^{2p+2}\right)  ...}{\left(
1-q^{\left(  \alpha+1\right)  (p+1)}\right)  \left(  1-q^{\left(
\alpha+2\right)  (p+1)}\right)  ...}\frac{(1-q^{\alpha+1})(1-q^{\alpha+2}%
)...}{(1-q)(1-q^{2})...}\\
&  =\frac{\left(  1-q^{p+1}\right)  _{q^{p+1}}^{\left(  \alpha\right)  }%
}{\left(  1-q\right)  ^{\left(  \alpha\right)  }}=\frac{\left(  1-q^{p+1}%
\right)  _{q^{p+1}}^{\left(  \alpha\right)  }}{\left(  1-q\right)  ^{\alpha
}\Gamma_{q}(\alpha+1)}=\frac{\Gamma_{q^{p+1}}(\alpha+1)}{\Gamma_{q}(\alpha
+1)}\left(  \left[  p+1\right]  _{q}\right)  ^{\alpha}%
\end{align}

\bigskip In addition, $q-$Beta function is defined as%

\[
\beta_{q}(t,s)=\int\limits_{0}^{1}\mathit{x}^{t-1}\left(  1-qx\right)
_{q}^{s-1}d_{q}x=\frac{\Gamma_{q}(t)\Gamma_{q}(s)}{\Gamma_{q}(t+s)}%
\]

\section{Itterated q-integral to approach new class of operators}

\bigskip There are several approaches to fractional differential operators.
One of demonstration methods of fractional differential equation is using the
itterated Cauchy integrals. The Riemann--Liouville fractional integral is a
generalization of the following itterated Cauchy integral:%

\[
\int\limits_{a}^{x}\mathit{dt}_{1}\int\limits_{a}^{\mathit{t}_{1}}%
\mathit{dt}_{2}...\int\limits_{a}^{\mathit{t}_{n-1}}f(t_{n})\mathit{dt}%
_{n}=\frac{1}{\Gamma(n)}\int\limits_{a}^{x}\left(  x-t\right)  ^{n-1}f(t)dt
\]

In the aid of this formula, for any positive real value $0<\alpha,$ we have%

\[
_{a}I^{\alpha}(f(x))=\frac{1}{\Gamma(\alpha)}\int\limits_{a}^{x}\left(
x-t\right)  ^{\alpha-1}f(t)\mathit{dt}%
\]

If we put $\frac{1}{t_{i}}$ in the chain of integration, then we can reach to
Hadamard operator. The related itterated integral can be written as%

\[
\int\limits_{a}^{x}\frac{1}{t_{1}}\mathit{dt}_{1}\int\limits_{a}%
^{\mathit{t}_{1}}\frac{1}{t_{2}}\mathit{dt}_{2}...\int\limits_{a}%
^{\mathit{t}_{n-1}}\frac{1}{t_{n}}f(t_{n})\mathit{dt}_{n}=\frac{1}{\Gamma
(n)}\int\limits_{a}^{x}\left(  Log\left(  \frac{x}{t}\right)  \right)
^{n-1}f(t)\frac{dt}{t}%
\]

Therefore, Hadamard fractional integral can be written as \cite{hadamard}%

\[
_{a}J^{\alpha}(f(x))=\frac{1}{\Gamma(\alpha)}\int\limits_{a}^{x}\left(
Log\left(  \frac{x}{t}\right)  \right)  ^{\alpha-1}f(t)\frac{dt}{t}%
\]

The author in \cite{katagumpala} assumed $t_{i}^{p}$ in the chain of
integration and reached to the general formula for fractional integral
operator.There are four different models of $q$-analogues of Riemann-Liouville
fractional integral operators. No one has investigated the Hadamard type and
there is not any $q$-analogue of this operator. First let us rewrite the
definition of $q-$fractional integral operator in all introduced forms. In
fact, for $\alpha\geq0$ and $f:\left[  a,b\right]  \rightarrow%
\mathbb{R}
,$ the $\alpha$ order fractional $q$-integral of a function $f(x)$ is defined
by
\[
I_{q,a}^{\alpha}(f(x))=\frac{1}{\Gamma_{q}(\alpha)}\int\limits_{a}^{x}%
K_{q}(t,x)f(t)\mathit{d}_{q}\mathit{t}%
\]

The kernel of this integral is defined as $K_{q}(t,x)=(x-qt)^{\left(
\alpha-1\right)  }$ \cite{rajkovic}, $K_{q}(t,x)=(x-qt)_{q}^{\alpha-1}$
\cite{gamma kac}, $K_{q}(t,x)=x^{\alpha-1}(\frac{qt}{x};q)_{\alpha-1}$
\cite{mansour}, and $K_{q}(t,x)=(x-qt)_{\alpha-1}$ \cite{agarwal}. There is a
good discussion about the form of the kernel function at \cite{Zhang}.

In this section, we investigate the new definition for q-integral operator. In
the aid of definition of Hadamard operator as itterated integral for the case
$\alpha\in%
\mathbb{N}
.$ We rewrite this integrals by using Jackson integral instead of Riemann
integral. We will use the technique of interchanging the order of q-integral
which is mentioned in the last section. So for $n=2,$ we have:
\[
J_{p,q}^{2}\left(  f(a)\right)  =\int\limits_{0}^{a}\int\limits_{0}%
^{\mathit{x}}\mathit{x}^{p}y^{p}f(y)\mathit{d}_{q}\mathit{yd}_{q}%
\mathit{x}=\int\limits_{0}^{a}\int\limits_{qy}^{\mathit{a}}\mathit{x}^{p}%
y^{p}f(y)\mathit{d}_{q}\mathit{xd}_{q}\mathit{y=}\frac{1}{\left[  p+1\right]
_{q}}\int\limits_{0}^{a}y^{p}f(y)\left[  a^{p+1}-q^{p+1}y^{p+1}\right]
\mathit{d}_{q}\mathit{y}%
\]

Now Consider the case $n=3$, then we have%

\begin{align*}
J_{p,q}^{3}\left(  f(a)\right)   &  =\int\limits_{0}^{a}\int\limits_{0}%
^{\mathit{x}}\int\limits_{0}^{\mathit{y}}\mathit{x}^{p}y^{p}z^{p}%
f(z)\mathit{d}_{q}\mathit{zd}_{q}\mathit{yd}_{q}\mathit{x}=\int\limits_{0}%
^{a}\int\limits_{qz}^{\mathit{a}}\int\limits_{qz}^{\mathit{x}}\mathit{x}%
^{p}y^{p}z^{p}f(z)\mathit{d}_{q}\mathit{yd}_{q}\mathit{x\mathit{d}%
_{q}\mathit{z}}\\
&  =\frac{1}{\left[  p+1\right]  _{q}\left[  2p+2\right]  _{q}}\int
\limits_{0}^{a}z^{p}f(z)\left[  \left(  a^{p+1}\right)  ^{2}-\frac{\left[
2p+2\right]  _{q}}{\left[  p+1\right]  _{q}}\left(  zq\right)  ^{p+1}%
a^{p+1}+\left(  \frac{\left[  2p+2\right]  _{q}}{\left[  p+1\right]  _{q}%
}-1\right)  \left(  \left(  zq\right)  ^{p+1}\right)  ^{2}\right]
\mathit{d}_{q}\mathit{z}%
\end{align*}

we can simplify this with the little calculation as%

\[
J_{p,q}^{3}\left(  f(a)\right)  =\mathit{=}\frac{1}{\left[  p+1\right]
_{q}\left[  2p+2\right]  _{q}}\int\limits_{0}^{a}z^{p}f(z)\left(
a^{p+1}-\left(  zq\right)  ^{p+1}\right)  \left(  a^{p+1}-\left(  zq\right)
^{p+1}q^{p+1}\right)  \mathit{d}_{q}\mathit{z}%
\]

For case $n=4$, we have%

\begin{align*}
J_{p,q}^{4}\left(  f(a)\right)   &  =\int\limits_{0}^{a}\int\limits_{0}%
^{\mathit{x}}\int\limits_{0}^{\mathit{y}}\int\limits_{0}^{\mathit{z}%
}\mathit{x}^{p}y^{p}z^{p}w^{p}f(w)\mathit{d}_{q}\mathit{wd}_{q}\mathit{zd}%
_{q}\mathit{yd}_{q}\mathit{x}=\int\limits_{0}^{a}\int\limits_{qw}^{\mathit{a}%
}\int\limits_{qw}^{\mathit{x}}\int\limits_{qw}^{\mathit{y}}\mathit{x}^{p}%
y^{p}z^{p}w^{p}f(w)\mathit{d}_{q}\mathit{yd}_{q}\mathit{x\mathit{d}%
_{q}\mathit{zd}_{q}\mathit{w}}\\
&  =\frac{1}{\left[  p+1\right]  _{q}\left[  2p+2\right]  _{q}\left[
3p+3\right]  _{q}}\int\limits_{0}^{a}w^{p}f(w)\left[  \left(  a^{p+1}\right)
^{3}-\left(  1+q^{p+1}+q^{2p+2}\right)  \left(  wq\right)  ^{p+1}\left(
a^{p+1}\right)  ^{2}+\right. \\
&  \left.  \left(  q^{p+1}+q^{2p+2}+q^{3p+3}\right)  \left(  \left(
wq\right)  ^{p+1}\right)  ^{2}a^{p+1}-q^{3p+3}\left(  \left(  wq\right)
^{p+1}\right)  ^{3}\right]  \mathit{d}_{q}\mathit{w}%
\end{align*}

Now the little calculation, shows that%

\begin{align*}
J_{p,q}^{4}\left(  f(a)\right)   &  =\frac{1}{\left[  p+1\right]  _{q}\left[
2p+2\right]  _{q}\left[  3p+3\right]  _{q}}\int\limits_{0}^{a}w^{p}f(w)\left(
a^{p+1}-\left(  wq\right)  ^{p+1}\right)  \times\\
&  \left(  a^{p+1}-\left(  wq\right)  ^{p+1}q^{p+1}\right)  \left(
a^{p+1}-\left(  wq\right)  ^{p+1}q^{2p+2}\right)  \mathit{d}_{q}\mathit{w}%
\end{align*}

We can easily see that for any natural number $k\in%
\mathbb{N}
,$ in the aid of induction we have:%

\[
J_{p,q}^{k}\left(  f(a)\right)  =\frac{1}{\prod_{n=1}^{k-1}\left[  n\left(
p+1\right)  \right]  _{q}}\int\limits_{0}^{a}w^{p}f(w)\prod_{n=0}^{k-1}\left(
a^{p+1}-\left(  wq\right)  ^{p+1}q^{n\left(  p+1\right)  }\right)
\mathit{d}_{q}\mathit{w}%
\]

This relation for itteration integral motivate us to define $q$-analogue of
integral operator as follow

\begin{definition}
for $\alpha>0$ and $x>0,$ if $f(x)$ satisfies the condistion of existance for
following Jackson integral, we define $q$-fractional integral as%

\begin{align*}
J_{p,q}^{\alpha}\left(  f(a)\right)   &  =\frac{1}{\left[  p+1\right]
^{\left(  \alpha-1\right)  }\Gamma_{q}(\alpha)}\int\limits_{0}^{a}%
w^{p}f(w)(a^{p+1}-\left(  wq\right)  ^{p+1})_{q^{p+1}}^{\left(  \alpha
-1\right)  }\mathit{d}_{q}\mathit{w}\\
&  \mathit{=}\frac{\left(  1-q\right)  ^{\alpha-1}}{\left(  1-q^{p+1}\right)
_{q^{p+1}}^{\left(  \alpha-1\right)  }}\int\limits_{0}^{a}w^{p}f(w)(a^{p+1}%
-\left(  wq\right)  ^{p+1})_{q^{p+1}}^{\left(  \alpha-1\right)  }%
\mathit{d}_{q}\mathit{w}\\
&  \mathit{=}\frac{\left(  \left[  p+1\right]  _{q}\right)  ^{1-\alpha}%
}{\Gamma_{q^{p+1}}(\alpha)}\int\limits_{0}^{a}w^{p}f(w)(a^{p+1}-\left(
wq\right)  ^{p+1})_{q^{p+1}}^{\left(  \alpha-1\right)  }\mathit{d}%
_{q}\mathit{w}%
\end{align*}

\begin{remark}
First, we should clear some identities. For any natural number $k\in%
\mathbb{N}
$ we have:
\[
\left[  k\left(  p+1\right)  \right]  _{q}=\frac{1-q^{k\left(  p+1\right)  }%
}{1-q}=\frac{1-q^{k\left(  p+1\right)  }}{1-q^{k}}\frac{1-q^{k}}{1-q}=\left[
p+1\right]  _{q^{k}}\left[  k\right]  _{q}%
\]

In addition, this definition is really the unification of $q$-analogue of
Reimann and Hadamard integral operator. For instance, let $q\rightarrow1^{-}$
then we have%

\[
\lim_{q\rightarrow1^{-}}J_{p,q}^{\alpha}\left(  f(a)\right)  =\frac{\left(
p+1\right)  ^{1-\alpha}}{\Gamma(\alpha)}\int\limits_{0}^{a}w^{p}f(w)\left(
a^{p+1}-w^{p+1}\right)  ^{\alpha-1}\mathit{dw}%
\]

This is exactly as the same as operator which is introduced at
\cite{katagumpala}. If we let $p\rightarrow-1^{+}$ and use Hopital, we have%

\[
\lim_{p\rightarrow-1^{+}}\lim_{q\rightarrow1^{-}}J_{p,q}^{\alpha}\left(
f(a)\right)  =\frac{1}{\Gamma(\alpha)}\int\limits_{0}^{a}\lim_{p\rightarrow
-1^{+}}\left(  \frac{a^{p+1}-w^{p+1}}{p+1}\right)  ^{\alpha-1}w^{p}%
f(w)\mathit{dw=}\frac{1}{\Gamma(\alpha)}\int\limits_{0}^{a}\left(  Log\left(
\frac{a}{w}\right)  \right)  ^{\alpha-1}f(w)\frac{dw}{w}%
\]

When $p=0$, we arrive at the well-known $q-$fractional Reimann \ integral
\cite{rajkovic}.
\end{remark}

\begin{remark}
The new operator can be expanded as%
\[
J_{p,q}^{\alpha}\left(  f(x)\right)  =\frac{x^{\alpha(p+1)}}{\left(
1-q\right)  ^{\alpha-2}\left(  1-q^{p+1}\right)  _{q^{p+1}}^{\left(
\alpha-1\right)  }}%
{\displaystyle\sum\limits_{i=0}^{\infty}}
\left(  q^{p+1}\right)  ^{i}f(xq^{i})(1-\left(  q^{i+1}\right)  ^{p+1}%
)_{q^{p+1}}^{\left(  \alpha-1\right)  }%
\]

In addition, we may express the $q$-Reimann type operator for $q^{p+1}$ as
follow
\[
I_{q^{p+1}}^{\alpha}\left(  f(x)\right)  =\frac{\left(  1-q\right)  ^{\alpha
}x^{\alpha}}{\left(  1-q^{p+1}\right)  _{q^{p+1}}^{\left(  \alpha-1\right)  }}%
{\displaystyle\sum\limits_{i=0}^{\infty}}
\left(  q^{p+1}\right)  ^{i}f(xq^{i(p+1)})(1-\left(  q^{i+1}\right)
^{p+1})_{q^{p+1}}^{\left(  \alpha-1\right)  }%
\]

These expressions shows that the new operator is not only the simple
modification of the available operator.
\end{remark}
\end{definition}

\section{Some properties of new q-fractional integral operator}

In this section, we study some familar properties of fractional integral
operator as semi-group properties of this operator. This property is
essentially useful, because to solve related $q-$difference equation we should
apply this property. In addition, we will define inverse operator as
$q$-fractional derivative and at the end, properties of these operators will
be studied. In this procedure, we study some useful identities and relations
as well. $q-$fractional Reimann integral operators were extensively
investigated in several resources.\cite{rajkovic} In the aid of Hine's
transform for q-hypergeometric functions, useful identities were introduced
and a lot of identities were studied.\cite{rajkovic} \cite{mansour} We start
with the following lemma from \cite{rajkovic}. This relation is important and
make a rule like beta function. Most of $q$-analogue of ordinary cases can be
interpreted in the aid of this lemma, We will introduce the sequence of
identities to achive semi-group property.

\begin{lemma}
For $\alpha,\beta,\mu\in%
\mathbb{R}
^{+}$, the following identity is valid\cite{rajkovic}
\end{lemma}

\[
\sum\limits_{t=0}^{\infty}\frac{(1-\mu q^{1-t})^{\left(  \alpha-1\right)
}(1-q^{1+t})^{\left(  \beta-1\right)  }}{(1-q)^{\left(  \alpha-1\right)
}(1-q)^{\left(  \beta-1\right)  }}\left(  q^{t}\right)  ^{\alpha}=\frac{(1-\mu
q)^{\left(  \alpha+\beta-1\right)  }}{(1-q)^{\left(  \alpha+\beta-1\right)  }}%
\]

\bigskip

\begin{remark}
We mention that the terms at the doniminator of summation is not related to
$t$ and we can take it out from the summation. Now, we modify the lemma to
prepare the suitable conditions for using in our operator. For this reason,
let us substitute $q$ by $q^{p+1}$ in the last lemma. In the aid of definition
that is mentioned by identity (9), we have :%
\[
\sum\limits_{t=0}^{\infty}(1-\mu\left(  q^{p+1}\right)  ^{1-t})_{q^{p+1}%
}^{\left(  \alpha-1\right)  }(1-\left(  q^{p+1}\right)  ^{1+t})_{q^{p+1}%
}^{\left(  \beta-1\right)  }\left(  q^{1+p}\right)  ^{t\alpha}=\frac
{(1-q^{p+1})_{q^{p+1}}^{\left(  \alpha-1\right)  }(1-q^{p+1})_{q^{p+1}%
}^{\left(  \beta-1\right)  }}{(1-q^{p+1})_{q^{p+1}}^{\left(  \alpha
+\beta-1\right)  }}(1-\mu q^{p+1})_{q^{p+1}}^{\left(  \alpha+\beta-1\right)  }%
\]

\end{remark}

In this step, let us calculate the following q-integral by using the last remark.

\begin{lemma}
Following Jackson integral for real positive $\alpha$ and $\lambda>-1$ holds true:%

\[
\int\limits_{a}^{x}t^{p}(x^{p+1}-\left(  qt\right)  ^{p+1})_{q^{p+1}}^{\left(
\alpha-1\right)  }(t^{p+1}-a^{p+1})_{q^{p+1}}^{\left(  \lambda\right)  }%
d_{q}t=(1-q)\left(  \frac{(1-q^{p+1})_{q^{p+1}}^{\left(  \alpha-1\right)
}(1-q^{p+1})_{q^{p+1}}^{\left(  \lambda\right)  }}{(1-q^{p+1})_{q^{p+1}%
}^{\left(  \alpha+\lambda\right)  }}\right)  \left[  (x^{p+1}-a^{p+1}%
)_{q^{p+1}}^{\left(  \alpha+\lambda\right)  }\right]
\]

\begin{proof}
In the aid of definition of Jackson integral, left hand side of this
inequality can be written as%

\begin{align*}
&  \int\limits_{a}^{x}t^{p}(x^{p+1}-\left(  qt\right)  ^{p+1})_{q^{p+1}%
}^{\left(  \alpha-1\right)  }(t^{p+1}-a^{p+1})_{q^{p+1}}^{\left(
\lambda\right)  }d_{q}t\\
&  =\int\limits_{0}^{x}t^{p}(x^{p+1}-\left(  qt\right)  ^{p+1})_{q^{p+1}%
}^{\left(  \alpha-1\right)  }(t^{p+1}-a^{p+1})_{q^{p+1}}^{\left(
\lambda\right)  }d_{q}t-\int\limits_{0}^{a}t^{p}(x^{p+1}-\left(  qt\right)
^{p+1})_{q^{p+1}}^{\left(  \alpha-1\right)  }(t^{p+1}-a^{p+1})_{q^{p+1}%
}^{\left(  \lambda\right)  }d_{q}t
\end{align*}

Here, the second integral is zero, because for some $i\in%
\mathbb{N}
$ the factor $(\left(  aq^{i}\right)  ^{p+1}-a^{p+1})_{q^{p+1}}^{\left(
\lambda\right)  }=0$ and we can expand the integral as%

\begin{align*}
\int\limits_{0}^{a}t^{p}(x^{p+1}-\left(  qt\right)  ^{p+1})_{q^{p+1}}^{\left(
\alpha-1\right)  }(t^{p+1}-a^{p+1})_{q^{p+1}}^{\left(  \lambda\right)  }%
d_{q}t  &  =\\
a(1-q)\sum\limits_{i=0}^{\infty}q^{i}\left(  aq^{i}\right)  ^{p}%
(x^{p+1}-\left(  aq^{i+1}\right)  ^{p+1})_{q^{p+1}}^{\left(  \alpha-1\right)
}(\left(  aq^{i}\right)  ^{p+1}-a^{p+1})_{q^{p+1}}^{\left(  \lambda\right)  }
&  =0
\end{align*}

Now expand the first integral and use the last remark to see%

\begin{align*}
&  \int\limits_{0}^{x}t^{p}(x^{p+1}-\left(  qt\right)  ^{p+1})_{q^{p+1}%
}^{\left(  \alpha-1\right)  }(t^{p+1}-a^{p+1})_{q^{p+1}}^{\left(
\lambda\right)  }d_{q}t\\
&  =\left(  x^{p+1}\right)  ^{\alpha+\lambda}(1-q)\sum\limits_{i=0}^{\infty
}\left(  q^{i}\right)  ^{\left(  p+1\right)  \left(  \lambda+1\right)
}(1-\left(  q^{i+1}\right)  ^{p+1})_{q^{p+1}}^{\left(  \alpha-1\right)
}\left(  1-\left(  \frac{a}{xq}\right)  ^{p+1}\left(  q^{1-i}\right)
^{p+1}\right)  _{q^{p+1}}^{\left(  \lambda\right)  }\\
&  =\left(  x^{p+1}\right)  ^{\alpha+\lambda}(1-q)\frac{(1-q^{p+1})_{q^{p+1}%
}^{\left(  \alpha-1\right)  }(1-q^{p+1})_{q^{p+1}}^{\left(  \lambda\right)  }%
}{(1-q^{p+1})_{q^{p+1}}^{\left(  \alpha+\lambda\right)  }}\left(  1-\left(
\frac{a}{xq}\right)  ^{p+1}q^{p+1}\right)  _{q^{p+1}}^{\left(  \lambda
+\alpha\right)  }\\
&  =(1-q)\left(  \frac{(1-q^{p+1})_{q^{p+1}}^{\left(  \alpha-1\right)
}(1-q^{p+1})_{q^{p+1}}^{\left(  \lambda\right)  }}{(1-q^{p+1})_{q^{p+1}%
}^{\left(  \alpha+\lambda\right)  }}\right)  \left(  x^{p+1}-a^{p+1}\right)
_{q^{p+1}}^{\left(  \lambda+\alpha\right)  }%
\end{align*}

\begin{remark}
put $a=0$ in the last integral to find $J_{p,q}^{\alpha}\left(  f(t)\right)  $
where $f(t)=t^{\lambda(p+1)}$. In the aid of last lemma we have%
\[
J_{p,q}^{\alpha}\left(  t^{\lambda(p+1)}\right)  =\frac{\Gamma_{q^{p+1}%
}(\alpha)\Gamma_{q^{p+1}}(\lambda+1)}{\left[  p+1\right]  _{q}\Gamma_{q^{p+1}%
}(\lambda+\alpha+1)}x^{(p+1)\left(  \lambda+\alpha\right)  }%
\]

In addition, we interpret logarithm function by limit of expression in remark
3. Hadamard integral operator has the following property: \cite{seri}%
\[
J_{a^{+}}^{\alpha}\left(  \left(  \log\left(  \frac{t}{a}\right)  \right)
^{\lambda}\right)  \left(  x\right)  =\frac{\Gamma(\lambda+1)}{\Gamma
(\lambda+\alpha+1)}\left(  \log\left(  \frac{t}{a}\right)  \right)
^{\lambda+\alpha}%
\]

On the other hand, we have this limit%
\[
\lim_{p\rightarrow-1^{+}}\lim_{q\rightarrow1^{-}}\left(  \frac{(t^{p+1}%
-a^{p+1})_{q^{p+1}}^{\left(  \lambda\right)  }}{\left[  p+1\right]  ^{\left(
\lambda\right)  }}\right)  =\left(  \log\left(  \frac{t}{a}\right)  \right)
^{\lambda}%
\]

Now, in the aid of last lemma, we derive to q-analogue of the property in
\cite{seri}%
\[
J_{a^{+},p,q}^{\alpha}\left(  \frac{(t^{p+1}-a^{p+1})_{q^{p+1}}^{\left(
\lambda\right)  }}{\left[  p+1\right]  ^{\left(  \lambda\right)  }}\right)
=\frac{(1-q)_{q^{p+1}}^{\left(  \alpha-1\right)  }\Gamma_{q}(\lambda
+1)}{(1-q)^{\alpha-1}\Gamma_{q}(\lambda+\alpha+1)\left[  p+1\right]  ^{\left(
\alpha+\lambda\right)  }}\left(  x^{p+1}-a^{p+1}\right)  _{q^{p+1}}^{\left(
\lambda+\alpha\right)  }%
\]

\end{remark}

\begin{proposition}
The given q-fractional integral operator has semi-group property. Means%

\[
J_{p,q}^{\alpha}\left(  J_{p,q}^{\beta}f(x)\right)  =J_{p,q}^{\alpha+\beta
}f(x)
\]

\begin{proof}
Write the left hand side of this identity as%

\begin{align*}
J_{p,q}^{\alpha}\left(  J_{p,q}^{\beta}f(x)\right)   &  =\frac{\left(
1-q\right)  ^{\alpha-1}}{\left(  1-q^{p+1}\right)  _{q^{p+1}}^{\left(
\alpha-1\right)  }}\int\limits_{0}^{x}w^{p}(x^{p+1}-\left(  wq\right)
^{p+1})_{q^{p+1}}^{\left(  \alpha-1\right)  }\left(  J_{p,q}^{\beta
}f(w)\right)  \mathit{d}_{q}\mathit{w}\\
&  =\frac{\left(  1-q\right)  ^{\alpha+\beta-2}}{\left(  1-q^{p+1}\right)
_{q^{p+1}}^{\left(  \alpha-1\right)  }\left(  1-q^{p+1}\right)  _{q^{p+1}%
}^{\left(  \beta-1\right)  }}\int\limits_{0}^{x}w^{p}(x^{p+1}-\left(
wq\right)  ^{p+1})_{q^{p+1}}^{\left(  \alpha-1\right)  }\\
&  \times\left(  \int\limits_{0}^{w}s^{p}f(s)(w^{p+1}-\left(  sq\right)
^{p+1})_{q^{p+1}}^{\left(  \beta-1\right)  }\mathit{d}_{q}\mathit{s}\right)
\mathit{d}_{q}\mathit{w}\\
&  =\frac{\left(  1-q\right)  ^{\alpha+\beta-2}}{\left(  1-q^{p+1}\right)
_{q^{p+1}}^{\left(  \alpha-1\right)  }\left(  1-q^{p+1}\right)  _{q^{p+1}%
}^{\left(  \beta-1\right)  }}\int\limits_{0}^{x}s^{p}f(s)\\
&  \times\left(  \int\limits_{qs}^{x}w^{p}(x^{p+1}-\left(  wq\right)
^{p+1})_{q^{p+1}}^{\left(  \alpha-1\right)  }(w^{p+1}-\left(  sq\right)
^{p+1})_{q^{p+1}}^{\left(  \beta-1\right)  }\mathit{d}_{q}\mathit{w}\right)
\mathit{d}_{q}\mathit{s}%
\end{align*}

Now apply the last lemma to have%

\begin{align*}
J_{p,q}^{\alpha}\left(  J_{p,q}^{\beta}f(x)\right)   &  =\frac{\left(
1-q\right)  ^{\alpha+\beta-2}}{\left(  1-q^{p+1}\right)  _{q^{p+1}}^{\left(
\alpha-1\right)  }\left(  1-q^{p+1}\right)  _{q^{p+1}}^{\left(  \beta
-1\right)  }}\int\limits_{0}^{x}s^{p}f(s)\\
&  \times\left(  (1-q)\left(  \frac{(1-q^{p+1})_{q^{p+1}}^{\left(
\alpha-1\right)  }(1-q^{p+1})_{q^{p+1}}^{\left(  \beta-1\right)  }}%
{(1-q^{p+1})_{q^{p+1}}^{\left(  \alpha+\beta-1\right)  }}\right)  \left[
(x^{p+1}-\left(  sq\right)  ^{p+1})_{q^{p+1}}^{\left(  \alpha+\beta-1\right)
}\right]  \right)  \mathit{d}_{q}\mathit{s}\\
&  \mathit{=}\frac{\left(  1-q\right)  ^{\alpha+\beta-1}}{(1-q^{p+1}%
)_{q^{p+1}}^{\left(  \alpha+\beta-1\right)  }}\int\limits_{0}^{x}%
s^{p}f(s)(x^{p+1}-\left(  sq\right)  ^{p+1})_{q^{p+1}}^{\left(  \alpha
+\beta-1\right)  }\mathit{d}_{q}\mathit{s=}J_{p,q}^{\alpha+\beta}f(x)
\end{align*}

\end{proof}

\begin{definition}
Let $\alpha\geq0$ and $n=\left\lfloor \alpha\right\rfloor +1$ means $n$ is the
smallest integer such that $n\geq\alpha$ and $p>0$. The corresponding
generalized $q-$fractional derivatives is defined by%
\begin{align*}
(D_{p,q}^{0}f)(x)  &  =f(x)\\
(D_{p,q}^{\alpha}f)(x)  &  =(x^{-p}D_{q})^{n}\left(  J_{p,q}^{n-\alpha
}\right)  f(x)=\frac{\left(  \left[  p+1\right]  _{q}\right)  ^{\alpha-n+1}%
}{\Gamma_{q^{p+1}}(n-\alpha)}(x^{-p}D_{q})^{n}\int\limits_{0}^{x}%
w^{p}f(w)(x^{p+1}-\left(  wq\right)  ^{p+1})_{q^{p+1}}^{\left(  n-\alpha
-1\right)  }\mathit{d}_{q}\mathit{w}\text{ \ \ \ \ }%
\end{align*}
if the integral does exist.
\end{definition}
\end{proposition}
\end{proof}
\end{lemma}

\-Now we can related the defined q-derivative and q-integral operator as follow%

\[
(D_{p,q}^{\alpha}J_{p,q}^{\alpha}f)(t)\text{ }=(x^{-p}D_{q})^{n}\left(
J_{p,q}^{n-\alpha}\right)  (J_{p,q}^{\alpha}f)(t)=(x^{-p}D_{q})^{n}%
(J_{p,q}^{n}f)(t)
\]
It is easy to see that $(x^{-p}D_{q})^{n}(J_{p,q}^{n}f)(t)=f(t)$. We can prove
it by induction. For instance, let us consider the case that $0<\alpha<1$ in
next proposition:

\begin{proposition}
\bigskip Assume that $0<\alpha<1$ and $p>0$ and integral does exist, then the
following identity holds%
\[
(D_{p,q}^{\alpha}J_{p,q}^{\alpha}f)(x)\text{ }=f(x)
\]

\begin{proof}
Direct calculation of the identity in the aid of lemma (5) shows that%
\begin{align*}
(D_{p,q}^{\alpha}J_{p,q}^{\alpha}f)(x)  &  =\frac{\left(  \left[  p+1\right]
_{q}\right)  }{\Gamma_{q^{p+1}}(\alpha)\Gamma_{q^{p+1}}(1-\alpha)}(x^{-p}%
D_{q})\int\limits_{0}^{x}\int\limits_{0}^{w}w^{p}s^{p}f(s)(w^{p+1}-\left(
sq\right)  ^{p+1})_{q^{p+1}}^{\left(  \alpha-1\right)  }\times\\
&  (x^{p+1}-\left(  wq\right)  ^{p+1})_{q^{p+1}}^{\left(  -\alpha\right)
}\mathit{d}_{q}\mathit{sd}_{q}\mathit{w}\\
&  =\frac{\left(  \left[  p+1\right]  _{q}\right)  }{\Gamma_{q^{p+1}}%
(\alpha)\Gamma_{q^{p+1}}(1-\alpha)}(x^{-p}D_{q})\int\limits_{0}^{x}%
s^{p}f(s)\times\\
&  \left(  \int\limits_{qs}^{x}w^{p}(w^{p+1}-\left(  sq\right)  ^{p+1}%
)_{q^{p+1}}^{\left(  \alpha-1\right)  }(x^{p+1}-\left(  wq\right)
^{p+1})_{q^{p+1}}^{\left(  -\alpha\right)  }\mathit{d}_{q}\mathit{w}\right)
\mathit{d}_{q}\mathit{s}\\
&  =\frac{\left(  \left[  p+1\right]  _{q}\right)  }{\Gamma_{q^{p+1}}%
(\alpha)\Gamma_{q^{p+1}}(1-\alpha)}(x^{-p}D_{q})\int\limits_{0}^{x}%
s^{p}f(s)\left(  \frac{\Gamma_{q^{p+1}}(\alpha)\Gamma_{q^{p+1}}(1-\alpha
)}{\left[  p+1\right]  _{q}}\right)  \mathit{d}_{q}\mathit{s=f(x)}%
\end{align*}

\end{proof}
\end{proposition}

In this paper, we defined class of generalized $q-$fractional difference
integral operator and the inverse operator also is defined. A few properties
of these operators were investigated, but still there are a lot of identities
and formulae related to this operator which can be studied. $q$-calculus is
the world of mathematics without limit and the introduced operator can be make
a rule as a part of these objects.


\begin{thebibliography}{99}                                                                                               %


\bibitem {kac}Cheung, P., \& Kac, V. G. (2001). Quantum calculus. Heidelberg: Springer.

\bibitem {gamma kac}De Sole, A., \& Kac, V. On integral representations of
q-gamma and q-beta functions. Rend. Mat. Acc. Lincei 9 (2005), 11-29. arXiv
preprint math.QA/0302032.

\bibitem {ernst}Ernst, T. (2000). The history of q-calculus and a new method.
Sweden: Department of Mathematics, Uppsala University.

\bibitem {mansour}Annaby, M. H., \& Mansour, Z. S. (2012). Q-fractional
Calculus and Equations (Vol. 2056). Springer.

\bibitem {unification}Mahmudov, N. I., \& Momenzadeh, M. (2018, August).
Unification of q--exponential function and related q-numbers and polynomials.
In AIP Conference Proceedings (Vol. 1997, No. 1, p. 020035). AIP Publishing.

\bibitem {hadamard}Hadamard, J. (1892). Essai sur l'\'{e}tude des fonctions,
donn\'{e}es par leur d\'{e}veloppement de Taylor. Gauthier-Villars.

\bibitem {katagumpala}Katugampola, U. N. (2011). New approach to a generalized
fractional integral. Applied Mathematics and Computation, 218(3), 860-865.

\bibitem {agarwal}Agarwal, R. P. (1969, September). Certain fractional
q-integrals and q-derivatives. In Mathematical Proceedings of the Cambridge
Philosophical Society (Vol. 66, No. 2, pp. 365-370). Cambridge University Press.

\bibitem {rajkovic}Rajkovi\'{c}, P. M., Marinkovi\'{c}, S. D., \&
Stankovi\'{c}, M. S. (2007). Fractional integrals and derivatives in
q-calculus. Applicable analysis and discrete mathematics, 311-323.

\bibitem {Zhang}Tang, Y., \& Zhang, T. (2019). A remark on the q-fractional
order differential equations. Applied Mathematics and Computation, 350, 198-208.

\bibitem {seri}Kilbas, A. A. A., Srivastava, H. M., \& Trujillo, J. J. (2006).
Theory and applications of fractional differential equations (Vol. 204).
Elsevier Science Limited.

\bibitem {alsalam}Al-Salam, W. A. (1966). q-Analogues of Cauchy's Formulas.
Proceedings of the American Mathematical Society, 17(3), 616-621.

\bibitem {11}Jia, B., Erbe, L., \& Peterson, A. (2019). Asymptotic behavior of
solutions of fractional nabla q-difference equations. Georgian Mathematical
Journal, 26(1), 21-28.

\bibitem {12}Noeiaghdam, Z., Allahviranloo, T., \& Nieto, J. J. (2019).
q-fractional differential equations with uncertainty. Soft Computing, 1-18.

\bibitem {log}Koelink, E., \& Van Assche, W. (2009). Leonhard Euler and a
$q$-analogue of the logarithm. Proceedings of the American Mathematical
Society, 137(5), 1663-1676.

\bibitem {jack}Jackson FH (1908) On q-functions and certain difference
operator. Trans R Soc Edinb 46:253--281

\bibitem {hadi}Ahmad, B., Alsaedi, A., Ntouyas, S. K., \& Tariboon, J. (2017).
Hadamard-type fractional differential equations, inclusions and inequalities.
Springer International Publishing.

\bibitem {applic}Goufo, E. F. D. (2015). A biomathematical view on the
fractional dynamics of cellulose degradation. Fractional Calculus and Applied
Analysis, 18(3), 554-564.

\bibitem {bop}Aral, A., Gupta, V., \& Agarwal, R. P. (2013). Applications of
q-calculus in operator theory (p. 262). New York: Springer.

\bibitem {bashir}Ahmad, B., Alsaedi, A., Ntouyas, S. K., \& Tariboon, J.
(2017). Hadamard-type fractional differential equations, inclusions and
inequalities. Springer International Publishing.
\end{thebibliography}
\end{document}